\documentclass[12pt,amscd]{amsart}
\usepackage[all]{xy}
\usepackage{graphicx}
\usepackage{amsmath,amsxtra,amssymb,latexsym, amscd, amsthm, amscd, enumerate}
\usepackage{indentfirst}
\usepackage[mathscr]{eucal}
\usepackage[unicode]{hyperref}
\usepackage{longtable}
\usepackage{stackrel}

\newtheorem{thm}{Theorem}[section]
\newtheorem{cor}[thm]{Corollary}
\newtheorem{lem}[thm]{Lemma}
\newtheorem{prop}[thm]{Proposition}
\theoremstyle{definition}
\newtheorem{defn}[thm]{Definition}
\newtheorem{exm}[thm]{Example}
\newtheorem{rem}[thm]{Remark}

\DeclareMathOperator{\lk}{lk}
\DeclareMathOperator{\st}{st}
\DeclareMathOperator{\core}{core}

\DeclareMathOperator{\im}{im}

\def\k{{\widetilde\chi}}
\def\h{{\widetilde H}}

\newcommand{\0}{\bold{0}}

\begin{document}

\title[A Characterization of Gorenstein Planar graphs] {A Characterization of Gorenstein Planar graphs}

\author{ Tran Nam Trung}
\address{Institute of Mathematics, VAST, 18 Hoang Quoc Viet, Hanoi, Viet Nam}
\email{tntrung@math.ac.vn}

\date{}
\dedicatory{}
\commby{}
\begin{abstract}

We prove that a planar graph is Gorenstein if and only if its independence complex is Eulerian.

\end{abstract}
\maketitle
\section*{Introduction}

Let $\Delta$ be a {\it simplicial complex} on the vertex set $V =\{1,\ldots,n\}$ where $n\geqslant 1$. Thus, $\Delta$ is a collection of subsets of $V$ closed under taking subsets; that is, if $\sigma \in \Delta$ and $\tau\subseteq \sigma$ then $\tau\in \Delta$. Let $k$ be a field and $R = k[x_1,\ldots,x_n]$ a polynomial ring over $k$. Let $I_{\Delta}$ be the ideal of $R$ generated all square-free monomials $x_{i_1}\cdots x_{i_s}$ such that $\{i_1,\ldots,i_s\}\notin \Delta$, i.e. a non-face of $\Delta$; this is called the {\it Stanley-Reisner} ideal of $\Delta$. Then, we say that $\Delta$ is {\it Cohen-Macaulay} (resp. {\it Gorenstein}) over $k$ if so is $R/I_{\Delta}$. One of the main problem is to characterize the Cohen-Macaulay and the Gorenstein property of $\Delta$ from the combinatorial data of $\Delta$ (see \cite{S}). For example,  let $|\Delta|$ denote the underlying topological space of $\Delta$, as defined in topology. Then, $\Delta$ is Gorenstein whenever $|\Delta|$ is isomorphic to a {\it sphere} according to \cite[Corollary 5.2]{S}. In general, these properties depend on the characteristic of  the base field $k$ (see \cite{R, GR}).

We now move on to {\it flag complexes}. A simplicial complex is a flag complex if all of its minimal non-faces are two element sets. Equivalently, if all of the edges of a potential face of a flag complex are in the complex, then that face must also be in the complex. Flag complexes are closely related to graphs. All graphs we consider in this paper are finite, not null and simple without isolated vertices. 

Let $G$ be a graph with vertex set $V(G) =\{1,\ldots,n\}$ and edge set $E(G)$. An {\it independent set} in $G$ is a set of vertices no two of which are adjacent to each other. The {\it independence complex} of  $G$, denoted by $\Delta(G)$, is the set of independent sets of $G$. We now associate to the graph $G$ a quadratic squarefree monomial ideal in $R$  $$I(G) = (x_ix_j  \mid ij \in E(G)),$$ which is called the {\it edge ideal} of $G$ (see \cite{Vi2}). We can see that
$$I_{\Delta(G)} = I(G),$$
and thus $\Delta(G)$ is a flag complex. Conversely, any flag complex is the independence complex of some graph. We say that $G$ is Cohen-Macaulay (resp. Gorenstein) over $k$ if so is $R/I(G)$. 

 A graph $G$ is called {\it well-covered} if very maximal independent set has the same size, that is $\alpha(G)$, the independence number of $G$. A well-covered graph $G$ is said to be a member of the class $W_2$ if $G\setminus v$ is well-covered with $\alpha(G\setminus v) = \alpha(G)$ for every vertex $v$ (see \cite{PL, Sp}). At first sight, if $G$ is Gorenstein then $G$ is in $W_2$ (see \cite[Lemma $2.4$]{HT}). In general we cannot read off the Gorenstein property of a graph just from its structure since this property as usual depends on the characteristic of $k$ (see \cite[Proposition $3.1$]{HT}), so now we focus on some classes of graphs such as (see \cite{HH,HHZ, HT}): 
\begin{enumerate}
\item A bipartite graph is Gorenstein if and only if it is consists of disjoint edges.
\item A chordal graph is Gorenstein if and only if it is consists of disjoint edges.
\item A triangle-free graph is Gorenstein if and only if it is a member of $W_2$.
\end{enumerate}

The main theme of this paper is to characterize Gorenstein graphs among {\it planar graphs}. A graph which can be drawn in the plane in such a way that edges meet only at points corresponding to their common ends is called a planar graph. In fact we deal with a larger class of graphs which we call {\it pseudo-planar graphs}. Before stating the result we need some terminology from Graph theory.  Two vertices which are incident with a common edge are {\it adjacent}, and two distinct adjacent vertices are {\it neighbors}. The set of neighbors of a vertex $v$ in a graph $G$ is denoted by $N_G(v)$. An {\it independent set} in $G$ is a set of vertices no two of which are adjacent to each other. The independence complex of  $G$, denoted by $\Delta(G)$, is the set of independent sets of $G$. For $S\in\Delta(G)$ we denote the neighbors of $S$ by 
$$N_G(S) := \{x \in V(G)\setminus S \mid N_G(x) \cap S \ne \emptyset\}$$
and the {\it localization} of $G$ with respect to $S$ by $G_S := G\setminus (S\cup N_G(S))$. Let $G[N_G(v)]$ be the induced subgraph of $G$ on the neighbors of $v$. When we remove all isolated vertices of $G[N_G(v)]$, the order of the remaining graph denoted by $\delta^*(v)$. Let $$\delta^*(G) :=\min\{\delta^*(v) \mid v \in V(G)\}.$$ 

The union of $m$ copies of disjoint graphs isomorphic to $G$ is denoted by $mG$. The {\it join} of two disjoint graphs $G$ and $H$ is denoted by $G * H$. The complete graph on $n$ vertices is denoted by $K_n$. In order to generalize the notion of triangle-free graphs and planar graphs we define {\it pseudo-planar graphs} as follows:

\begin{defn} A graph $G$ is pseudo-planar if every $S\in\Delta(G)$ satisfies:
\begin{enumerate}
\item $\delta^*(G_S) \leqslant 5$;
\item $G_S$ is not isomorphic to $2K_2 * 2K_2$.
\end{enumerate}
\end{defn}

Let $G$ be a graph and $\Delta := \Delta(G)$. From Lemmas $\ref{Reisner}$ and $\ref{Stanley}$, we see that  $G$ is Gorenstein if and only if:
\begin{enumerate}
\item [(i)]  For all $F\in \Delta$ and all $i < \dim(\lk_{\Delta} F)$, we have $\widetilde{H}_i(\lk_{\Delta} F; k) = \0$; and
\item [(ii)] $\Delta$ is Eulerian.
\end{enumerate}

The main result of the paper shows that the combinatorial condition (ii) implies the algebraic condition (i) whenever $G$ is pseudo-planar. In other words, we have a characterization of pseudo-planar graphs (and so of planar graphs, see Remark $\ref{PG}$):

\medskip

\noindent{\bf Theorem \ref{GlobalVersion}} {\it Let $G$ be a pseudo-planar graph. Then, $G$ is Gorenstein if and only if $\Delta(G)$ is Eulerian.
}

\medskip

The paper is organized as follows. In the section $1$ we set up some basic notations, terminology for graph and the simplicial complex. In the section $2$, we compute the Euler characteristics of the independence complexes of locally Gorenstein graphs. In the last section we characterize Gorenstein pseudo-planar graphs. 

\section{Preliminaries}

Let $\Delta$ be a simplicial complex on the vertex set $V(\Delta)= \{1,\ldots,n\}$. We define the {\it Stanley-Reisner} ideal of the simplicial complex $\Delta$ to be the squarefree monomial ideal
$$I_{\Delta} = (x_{j_1} \cdots x_{j_i} \mid j_1  <\cdots < j_i \ \text{ and } \{j_1,\ldots,j_i\} \notin \Delta) \ \text{ in } R = k[x_1,\ldots,x_n]$$
and the {\it Stanley-Reisner} ring of $\Delta$ to be the quotient ring $k[\Delta] = R/I_{\Delta}$. Then, $\Delta$ is Cohen-Macaulay (resp. Gorenstein) if $k[\Delta]$ is Cohen-Macaulay (resp. Gorenstein). 

If $F\in\Delta$, we define the dimension of $F\in\Delta$ to be $\dim F = |F| -1$ and the dimension of $\Delta$ to be $\dim\Delta = \max\{\dim F \mid F \in \Delta\}$. The link of $F$ inside $\Delta$ is its subcomplex:
$$\lk_{\Delta}  F = \{H\in \Delta \mid H\cup F\in \Delta \ \text{ and } H\cap F=\emptyset\}.$$

For each $i$, let $\mathcal {\widetilde C}_i(\Delta;k)$ the vector space over $k$ whose basis elements are the exterior products
 $e_F = e_{j_0} \wedge \cdots \wedge e_{j_i}$ that correspond to $i$-faces $F = \{j_0,\ldots,j_i\}\in \Delta$ with $j_0 < \cdots < j_i$. The reduced chain complex of $\Delta$ over $k$ is the complex $\mathcal {\widetilde C}_{\bullet}(\Delta;k)$ whose  differentials $\partial_i: \mathcal {\widetilde C}_i(\Delta;k) \ \longrightarrow  \mathcal {\widetilde C}_{i-1}(\Delta;k)$  is given by
$$ \partial_i(\ e_{j_0}\wedge \cdots\wedge e_{j_i}) = \sum_{s=0}^i (-1)^se_{j_0} \wedge \cdots \wedge \widehat{e}_{j_s} \wedge \cdots \wedge e_{j_i},$$
and the $i$-th homology group of $\Delta$ is $\h_i(\Delta;k)=\ker(\partial_{i})/ \im(\partial_{i+1})$. For simplicity, if $\omega\in \mathcal {\widetilde C}_i(\Delta;k)$, we write $\partial \omega$ stands for $\partial_i \omega$.  With this notation we have
\begin{equation}\label{DF}\partial  (\omega \wedge \nu) =\partial  \omega \wedge \nu + (-1)^{i+1} \omega \wedge \partial  \nu \ \text{ for all } \omega \in
\mathcal {\widetilde C}_i(\Delta;k) \text{ and } \nu \in \mathcal {\widetilde C}_j(\Delta;k).\end{equation}

The most widely used criterion for determining when a simplicial complex is Cohen-Macaulay is due to Reisner (see \cite[Corollary $4.2$]{S}), which says that links have only top homology.

\begin{lem} \label{Reisner} $\Delta$ is Cohen-Macaulay over $k$ if and only if for all $F\in \Delta$ and all $i < \dim(\lk_{\Delta} F)$, we have $\widetilde{H}_i(\lk_{\Delta} F; k) = \0$.
\end{lem}

Let $f_i$ be the number of $i$-dimensional faces of $\Delta$. Then the reduced Euler characteristic $\k(\Delta)$ of $\Delta$ is defined by
$$\k(\Delta) :=\sum_{i=-1}^d (-1)^i f_i = \sum_{F\in\Delta} (-1)^{|F|-1},$$
where $d=\dim(\Delta)$. This number can be represented via the reduced homology groups of $\Delta$ by:
$$\k(\Delta) = \sum_{i=-1}^d (-1)^i\dim_k\h_{i}(\Delta;k).$$

The restriction of $\Delta$ to a subset $S$ of $V(\Delta$) is $\Delta|_{S} :=\{F\in\Delta \mid F \subseteq S\}$. The star of a vertex $v$ in $\Delta$ is $\st_{\Delta}(v) :=\{F\in\Delta \mid F\cup \{v\}\in\Delta\}$. Let $\core(V(\Delta)) := \{x\in V(\Delta) \mid \st_{\Delta}(x) \ne V(\Delta)\}$, then the core of $\Delta$ is $\core(\Delta):= \Delta|_{\core(V(\Delta))}$. If $\Delta =\st_{\Delta}(v)$ for some vertex $v$, then $\Delta$ is a cone over $v$. Thus $\Delta=\core(\Delta)$ means $\Delta$ is not a cone.

Let $\Delta$ be a pure simplicial complex, i.e. every facet of $\Delta$ has the same cardinality. We say that $\Delta$ is an {\it Euler complex} if
$$\k(\lk_{\Delta} F) = (-1)^{\dim \lk_{\Delta} F} \text{ for all } F\in \Delta;$$ 
and $\Delta$ is a {\it semi-Euler complex} if $\lk_{\Delta}(x)$ is an Euler complex for all vertex $x$.

We then have a criterion for determining when Cohen-Macaulay complexes $\Delta$ are Gorenstein due to Stanley (see \cite[Theorem $5.1$]{S}).

\begin{lem} \label{Stanley} $\Delta$ is Gorenstein if and only if and only if $\core(\Delta)$ is an Euler complex which is Cohen-Macaulay.
\end{lem}

Let $S$ be a subset of the vertex set of $\Delta$ and let $\Delta\setminus S:=\{F\in\Delta \ | F \cap S =\emptyset\}$, so that $\Delta\setminus S$ is  a subcomplex of $\Delta$. If $S =\{x\}$, then we write $\Delta\setminus x$ stands for $\Delta\setminus \{x\}$. Clearly, $\Delta\setminus x = \{F\in\Delta \ | \ x\notin F\}$.

The following lemma is the key to investigate Cohen-Macaulay simplicial complexes in the sequence of this paper (see \cite[Lemma $1.4$]{HT}).

\begin{lem} \label{ZV} Let $\Delta$ be a Gorenstein simplicial complex with $\Delta = \core(\Delta)$. If $S$ is a subset of $V(\Delta)$ such that $\Delta|_S$ is a cone, then $\h_i(\Delta \setminus S,k) = 0$ for all $i$.
\end{lem}

Let $G$ be a graph. Then, $G$ is well-covered if and only if $\Delta(G)$ is pure; and $\Delta(G) = \core(\Delta(G))$ if and only if $G$ has no isolated vertices. The independence number of $G$, denoted by $\alpha(G)$, is the cardinality of the largest independent set in $G$. Then, $\dim (\Delta(G)) = \alpha(G)-1$.

For any independent set $S$ of $G$ we have $\Delta(G_S) = \lk_{\Delta(G)}(S)$. Therefore, $G_S$ is Cohen-Macaulay (resp. Gorenstein) if so is $G$ by Lemma $\ref{Reisner}$ (resp. Lemma $\ref{Stanley}$).

We collect here some basic facts about well-covered graphs that  will be used in the sequel.

\begin{lem}  \cite[Lemma $1$]{FHN}\label{LC} If $G$ be a well-covered graph and $S$ is a an independent set of $G$ then $G_S$ is well-covered. Moreover, $\alpha(G_S) =\alpha(G)-|S|$.
\end{lem}

\begin{lem} \label{WLC} If a graph $G$ is in $W_2$ and $S$ is an independent set of $G$, then $G_S$ is in $W_2$ whenever  $|S| < \alpha(G)$.
\end{lem}
\begin{proof} Follows immediately by induction on $|S|$ and using \cite[Theorem $5$]{Pi}.
\end{proof}

We say that $G$ is locally Cohen-Macaulay (resp. Gorenstein) if $G_v$ is Cohen-Macaulay (resp. Gorenstein) for all vertices $v$.

\begin{lem}\cite[Lemma $2.3$]{HT} \label{ML} Let $G$ be a locally Gorenstein graph in $W_2$ and let $S$ be a nonempty independent set of $G$. Then we have $G_S$ is Gorenstein and $\Delta(G_S)$ is Eulerian with $\dim(\Delta(G_S))=\dim (\Delta(G))-|S|$.
\end{lem}

A vertex $v$ of a well-covered graph $G$ is called {\it extendable} if $G\setminus v$ is well-covered and $\alpha(G\setminus v) = \alpha(G)$. Thus, $G$ is in $W_2$ if and only if every its vertex is extendable.

\begin{lem}\label{E1}Let $G$ be a graph with at least two vertices. If $\Delta(G)$ is semi-Eulerian, then $G$ is in $W_2$. \end{lem}

\begin{proof} We prove the lemma by induction on $\alpha(G)$. If $\alpha(G)=1$, then $G$ is a complete graph, and then $G$ is in $W_2$ as $|V(G)| \geqslant 2$.

Assume that $\alpha(G)\geqslant 2$. Since $\Delta(G)$ is semi-Eulerian, $G$ is well-covered.  Assume on contrary that $G$ is not in $W_2$, so that $G$ has a nonextendable vertex, say $v$. By \cite[Corollary $2.1$]{FHN} there would be an nonempty independent set $S$ of $G$ such that $v$ is an isolated vertex of $G_S$. In particular, $\k(\Delta(G_S)) = 0$.

On the other hand, by Lemma $\ref{LC}$ we have $\alpha(G_S) = \alpha(G) - |S| < \alpha(G)$. Thus,  $\Delta(G_S)=\lk_{\Delta(G)}(S)$ is Eulerian because $\Delta(G)$ is semi-Eulerian. It follows that $$\k(\Delta(G_S)) = (-1)^{\dim \Delta(G_S)} \ne 0,$$ a contradiction. Hence, every vertex of $G$ is extendable, and hence $G$ is in $W_2$.
\end{proof}

\begin{rem}\label{PG} A planar graph is a pseudo-planar graph.
\end{rem}

In order to see this we recall a criterion on deciding whether a given graph is planar.  A {\it minor} of a graph $G$ is any graph obtainable from $G$ by means of a sequence of vertex and edge deletions and edge contractions. Alternatively, consider a partition $(V_0,V_1, \ldots,V_m)$ of $V(G)$ such that $G[V_i]$ is connected, $1 \leqslant i \leqslant m$, and let $H$ be the graph obtained from $G$ by deleting $V_0$ and shrinking each induced subgraph $G[V_i]$, $1 \leqslant i \leqslant m$, to a single vertex. Then any spanning subgraph of $H$ is a minor of $G$.  A minor which is isomorphic to $K_5$ or $K_{3,3}$ is called a Kuratowski minor. Then, Wagner's theorem (see \cite[Theorem $10.3$]{BM}) says that: {\it $G$ is planar if and only if $G$ has no Kuratowski minor}. 

If $G$ be a planar graph and $S\in \Delta(G)$, then $G_S$ is an induced subgraph of $G$, and hence it is planar. By \cite[Corollary $10.22$]{BM} we deduce that $\delta^*(G_S) \leqslant 5$. Moreover, if $G_S \cong 2K_2 * 2 K_2$, then we can see that $G_S$ would have a minor $K_{3,3}$, a contradiction. It follows that $G$ is pseudo-planar.

\section{Euler characteristics of semi-Eulerian independence complexes}

In this section we give a formula for computing $\k(\Delta(G))$ when $\Delta(G)$ is semi-Eulerian.

\begin{lem} \label{SumEulerCh} Let $\Delta$ be a simplicial complex. If $\Delta$ is not a void complex, then
$$\sum_{F\in\Delta}\k(\lk_{\Delta}(F)) = -1.$$
\end{lem}
\begin{proof} Let $d := \dim(\Delta)$. We prove the lemma by induction on $d$. If $d = -1$, then $\Delta = \{\emptyset\}$, and  the lemma holds for this case.

If $d = 0$, then $\Delta$ consists of isolated vertices, say $x_1,\ldots,x_n$, where $n=|V(G)|$. We have
$$\sum_{F\in\Delta}\k(\lk_{\Delta}(F)) = \k(\lk_{\Delta}(\{\emptyset\})) + \sum_{i=1}^n \k(\lk_{\Delta}(x_i)) = -1+n-n = -1,$$
and thus the lemma holds.

Assume that $d \geqslant 1$. Let $F_1,\ldots,F_t$ be the facets of $\Delta$. We may assume that
$\dim(F_i) = d$ for all $i=1,\ldots,s$; and $\dim(F_i) < d$ for all $i=s+1,\ldots,t$; with $1 \leqslant s \leqslant t$. We now proceed to prove by induction on $s$. If $s = 1$, let $\Lambda = \Delta \setminus \{F_1\}$. Because $\dim \Lambda = \dim \Delta - 1 = d-1$, by the induction hypothesis we have $$\sum_{F\in\Lambda}\k(\lk_{\Delta}(F)) = -1,$$ 
so that
\begin{align*}\sum_{F\in\Delta}\k(\lk_{\Delta}(F)) &= \sum_{F\in\Lambda}\k(\lk_{\Delta}(F)) + \sum_{F\subseteq F_1} (-1)^{|F_1| - |F|-1}=-1+ (-1)^{|F_1|} \sum_{F\subseteq F_1} (-1)^{|F|-1}\\
&=-1+(-1)^{|F_1|}\k(\left<F_1\right>),
\end{align*}
where $\left<F_1\right>$ is a simplex on the vertex set $F_1$. Note that $F_1 \ne \emptyset$, so $\k(\left<F_1\right>)=0$, and so $\sum_{F\in\Delta}\k(\lk_{\Delta}(F)=-1$.

Assume that $s \geqslant 2$. By the same way as the previous case , let $\Lambda := \Delta\setminus \{F_s\}$. Since $\dim \Lambda = \dim \Delta$ and $\Lambda$ has one facet less than $\Delta$, by the induction hypothesis on $s$ we have $\sum_{F\in\Delta}\k(\lk_{\Lambda}(F))=-1$. Therefore,
$$\sum_{F\in\Delta}\k(\lk_{\Delta}(F)) = \sum_{F\in\Lambda}\k(\lk_{\Delta}(F)) +(-1)^{|F_s|}\k(\left<F_s\right>) = -1.$$
\end{proof}

\begin{prop} \label{EulerCh} Let $G$ be a graph such that $\Delta(G)$ is semi-Eulerian. Let $v$ be a vertex of $G$ and $d:= \dim(\Delta(G))$. Then,
$$\k(\Delta(G)) = (-1)^d(1+\k(\Delta(G[N(v)]))).$$
\end{prop}
\begin{proof} We prove by induction on $d$. If $d = 0$, then $\alpha(G)=1$. In this case, $G$ is a complete graph, so $G=K_n$ where $n=|V(G)|$. Let $v$ be a vertex of $G$. Then, $G[N(v)]=K_{n-1}$. Hence,
$$\k(\Delta(G)) = -1+n \text{ and } \k(G[N(v)])=-1+(n-1)=-2+n,$$ 
and hence the proposition holds.

Assume that $d\geqslant 1$ so that $\alpha(G)\geqslant 2$. By Lemma $\ref{E1}$ we have $G$ is in $W_2$, in particular $G$ is well-covered. Let $\Delta := \Delta(G)$ and $A := N(v)$. Let
$$\Gamma := \{F\in \Delta \mid F\cap A \ne\emptyset\},$$
so that $\Delta$ can be partitioned into $\Delta= \st_{\Delta}(v) \cup \Gamma$. Note that $\k(\st_{\Delta}(v)) = 0$ because $\st_{\Delta}(v)$ is a cone over $v$. Thus,
\begin{align*}
\k(\Delta) = \sum_{F\in\Delta} (-1)^{|F|-1} &= \sum_{F\in\st_{\Delta}(v)} (-1)^{|F|-1}+\sum_{F\in\Gamma} (-1)^{|F|-1}\\
&= \k(\st_{\Delta}(v)) + \sum_{F\in\Gamma} (-1)^{|F|-1} = \sum_{F\in\Gamma} (-1)^{|F|-1}.
\end{align*}
Let $\Lambda := \Delta(G[N(v)])$ and let $\Omega := \Lambda \setminus \{\emptyset\}$. For each $S\in\Omega$, we define
$$g(S):=\sum_{F\in\Gamma, S \subseteq F} (-1)^{|F|-1}, \ \text{ and } \tau(S):=\sum_{F\in\Gamma, F \cap A = S} (-1)^{|F|-1}.$$
Then, $$\k(\Delta) = \sum_{U\in\Omega}\tau(S), \ \text{ and } \ g(S) = \sum_{F\in \Omega, S\subseteq F} \tau(F).$$

For every $S\in \Omega$, as $S$ is a nonempty face of $\Delta$ and $\Delta$ is semi-Eulerian, we have $\Delta(G_S)$ is Eulerian. Since $G$ is well-covered, $\alpha(G_S)=\alpha(G)-|S|$ by Lemma $\ref{LC}$, and so $\dim (\Delta(G_S)) =\alpha(G_S)-1 = d -|S|$. Hence, $\k(\Delta(G_S)) = (-1)^{d -|S|}$, and hence
\begin{align*}g(S) &= \sum_{F\in\Gamma, S \subseteq F} (-1)^{|F|-1} =\sum_{F\in\Delta, S \subseteq F} (-1)^{|F|-1}= \sum_{F\in\Delta(G_S)} (-1)^{|F|+|S|-1} \\
&= (-1)^{|S|} \sum_{F\in\Delta(G_S)} (-1)^{|F|-1} = (-1)^{|S|} \k(G_S) = (-1)^{|S|} (-1)^{d-|S|} =  (-1)^d.
\end{align*}

We now consider $\Omega$ as a poset with the partial order $\leqslant$ being inclusion. Then, $g(S)$ can be written as
$$g(S) = \sum_{F\in \Omega, F\geqslant S} \tau(F).$$

Let $\mu$ be the Mobius function of the poset $\Omega$. Then by Mobius inversion formula (see \cite[Proposition $3.7.2$]{S1}) we have
$$\tau(S) = \sum_{F\in \Omega, F\geqslant S} \mu(S,F)g(F) = \sum_{F\in \Omega, F\geqslant S} \mu(S,F) (-1)^d = (-1)^d \sum_{F\in \Omega, F\geqslant S} \mu(S,F).$$
Observe that if $S\leqslant F$ in $\Omega$, then every $T$ such that $S \subseteq T \subseteq F$ we have $T\in \Omega$ and $S\leqslant T\leqslant F$. Hence, $\mu(S,F) = (-1)^{|F| - |S|}$, and hence
\begin{align*}\tau(S) &= (-1)^d \sum_{F\in \Omega, F\geqslant S} (-1)^{|F|-|S|} = (-1)^d \sum_{F\in \Omega, F\geqslant S} (-1)^{|F|-|S|}\\
&= (-1)^d \sum_{F\in \Lambda, S \subseteq F} (-1)^{|F|-|S|}= -(-1)^d \sum_{F\in \lk_{\Lambda}(S)} (-1)^{|F|-1} = -(-1)^d\k(\lk_{\Lambda}(S)).
\end{align*}
Therefore,
\begin{align*} \k(\Delta) &= \sum_{S\in\Omega}\tau(S) = \sum_{S\in\Lambda, S\ne\emptyset}\tau(S) =-(-1)^d \sum_{S\in\Lambda,
S\ne\emptyset}\k(\lk_{\Lambda}(S))\\
&=(-1)^d (\lk_{\Lambda}(\emptyset) - \sum_{S\in\Lambda}\k(\lk_{\Lambda}(S))) = (-1)^d (\k(\Lambda)) -\sum_{S\in\Lambda}\k(\lk_{\Lambda}(S))).
\end{align*}
On the other hand, by Proposition $\ref{SumEulerCh}$ we have $\sum_{S\in\Lambda}\k(\lk_{\Lambda}(S)) = -1$, and thus
$$\k(\Delta) = (-1)^d (1+\k(\Delta(G[N(v)]))),$$
as required.
\end{proof}

As a consequence we have.

\begin{cor}\label{E2} Let $G$ be a graph such that $\Delta(G)$ is semi-Eulerian. Then, the following conditions are equivalent:
\begin{enumerate}
\item $\Delta(G)$ is Eulerian;
\item $\k(\Delta(G[N(v)])) = 0$ for every vertex $v$ of $G$;
\item $\k(\Delta(G[N(v)])) = 0$ for some vertex $v$ of $G$.
\end{enumerate}
\end{cor}
\begin{proof} Let $\Delta = \Delta(G)$ and $d:=\dim(\Delta(G))$. $(1)\Longrightarrow (2)$:  Since $\Delta$ is Eulerian, $\k(\Delta)=(-1)^d$. By Proposition $\ref{EulerCh}$ we have
$$\k(\Delta) = (-1)^d = (-1)^d(1+\k(\Delta(G[N(v)])),$$
so that $\k(\Delta(G[N(v)])) = 0$.

$(2)\Longrightarrow (3)$: Obviously. $(3) \Longrightarrow (1)$: Since $\Delta$ is semi-Eulerian, by Proposition $\ref{EulerCh}$ we have
$$\k(\Delta) = (-1)^d (1+\k(\Delta(G[N(v)])) = (-1)^d.$$
Together with the semi-Eulerian property of $\Delta$, it follows that $\Delta$ is Eulerian.
\end{proof}

\section{Gorenstein Pseudo-Planar Graphs}
	
In this section, we will prove that if $G$ is a pseudo-planar graph such that $\Delta(G)$ is Eulerian then $G$ is Gorenstein.  We start when $\alpha(G)$ is small.

\begin{lem} \label{A1} Let $G$ be a pseudo-planar graph with $\alpha(G) \leqslant 2$. If $\Delta(G)$ is Eulerian then $G$ is Gorenstein.
\end{lem}

\begin{proof} If $\alpha(G) = 1$, then $G \cong K_n$ where $n=|V(G)|$. In this case we have
$$\k(\Delta(G)) = -1+n.$$
Since $\Delta(G)$ is Eulerian, we get $n=2$. Hence, $G$ is Gorenstein.

Assume that $\alpha(G)=2$. By Lemma $\ref{Stanley}$, it remain to show that $G$ is Cohen-Macaulay. For any vertex $v$, note that $\alpha(G_v)=1$ and $\Delta(G_v)$ is Eulerian. Therefore, $G_v$ is Gorenstein as in the previous case. By Lemma $\ref{Reisner}$ we have $\Delta(G)$ is Cohen-Macaulay if
$$\h_{-1}(\Delta(G);k) = \h_0(\Delta(G);k) = 0.$$

Since $\Delta(G)$ has at least one vertex, we have $\h_{-1}(\Delta(G);k) = 0$. In order to prove $\h_{0}(\Delta(G);k) = 0$ it suffices to show $\Delta(G)$ is connected.

Assume on contrary that $\Delta(G)$ is disconnected. Then we would have two proper subgraphs of $G$, say $G_1$ and $G_2$, so that $\alpha(G_1)= \alpha(G_2)$ and $G = G_1 * G_2$.

Let $v$ be any vertex of $G$. Since $G_v$ is an edge, we have $\deg_G(v) = n-3$. Hence, $G$ is $(n-3)$-regular. As $G[N(v)]$ is connected and $G$ is pseudo-planar, we have $n-3\leqslant 5$, and hence $n\leqslant 8$.

For every vertex $u$ of $G_1$ we have $G_u = (G_1)_u$, therefore $(G_1)_u$ is an edge. It implies that $\upsilon(G_1)\geqslant 4$ and the equality occurs if and only if $G_1 \cong 2K_2$. Similarly, $\upsilon(G_2) \geqslant 4$ and the equality occurs if and only if $G_2 \cong 2K_2$.

Because $n = \upsilon(G) = \upsilon(G_1)+\upsilon(G_2)$ and $n\leqslant 8$, we deduce that $n=8$, $G_1 \cong 2K_2$ and $G_2\cong 2K_2$. Consequently, $G \cong 2K_2 * 2 K_2$. This contradicts the pseudo-planar property of $G$. 	Therefore, $\Delta(G)$ must be connected, as required.
\end{proof}

\begin{lem} \label{A2} Let $G$ be a locally Gorenstein graph in $W_2$, $X$ an independent set of $G$ and an integer $i <\dim(\Delta(G))$. Then for all $\omega\in \mathcal C_i(\Delta(G);k)$ with $\partial \omega = 0$ can be written as $\omega = \partial \eta + \omega'$ where
$\eta\in \mathcal C_{i+1}(\Delta(G);k)$ and $\omega' \in \mathcal C_i(\Delta(G)\setminus X;k)$.
\end{lem}
\begin{proof} The lemma is trivial if $X=\emptyset$, so we assume that $X\ne\emptyset$. Let $\Delta := \Delta(G)$. We represent $\omega$ as:
$$\omega = \sum_{S\subseteq X}e_S\wedge \omega_S$$
where $\omega_S\in \mathcal C_{i-|S|}(\Delta(G_S\setminus X);k)$.

If there is $\emptyset\ne F\subseteq X$ such that $\omega_F\ne 0$, then we take such an $F$  such that $|F|$ is maximal. By Equation $(\ref{DF})$ we get
$$\partial \omega = \sum_{S\subseteq X}\left(\partial e_S\wedge \omega_S  + (-1)^{|S|}e_S\wedge \partial\omega_S \right)= 0.$$
Which implies $\partial \omega_F=0$. 

Next, we claim that
\begin{equation} \label{EQ-01}
\h_{i-|F|}(\Delta(G_F\setminus X);k) =0.
\end{equation}

Indeed, we have $G_F$ is Gorenstein because $F\ne\emptyset$. By Lemma $\ref{WLC}$ we imply that $G_F$ has no isolated vertices, whence $\core(\Delta(G_F)) = \Delta(G_F)$. By Lemma $\ref{LC}$ we have $\alpha(G_F)=\alpha(G)-|F|$, so $\dim(\Delta(G_F)) = \dim(\Delta) - |F|$.  In particular, $i - |F| <\dim(\Delta(G_F))$. If $F=X$, then $G_F\setminus X = G_F$ is Cohen-Macaulay. By Lemma $\ref{Reisner}$ we have $\h_{i-|F|}(\Delta(G_F\setminus X);k)= \h_{i-|F|}(\Delta(G_F);k)=0$.

If $F$ is a proper subset of $X$, then $G_F\setminus X = G_F\setminus (X\setminus F)$. Since $X\setminus F$ is an independent set of $G_F$, by Lemma $\ref{ZV}$ we have $\h_{i-|F|}(\Delta(G_F\setminus X);k)=0$, as claimed.

We now prove the lemma. By Equation $(\ref{EQ-01})$, we get $\omega_F =\partial( \eta_F)$
 for some $\eta_F \in \mathcal C_{i-|F|+1}(\Delta(G_F\setminus X); k)$. Write $F=\{a_1,\ldots,a_s\}$, where $a_1 <\cdots<a_s$,  then
$$\partial e_F = \sum_{i=1}^s(-1)^{i-1}e_{a_1} \wedge \cdots \wedge \widehat{e}_{a_i} \wedge \cdots \wedge e_{a_s} = \sum_{i=1}^s (-1)^{i-1}e_{F\setminus\{a_i\}}.$$
By Equation $(\ref{DF})$ we have
$$\partial(e_F\wedge \eta_F) = \partial e_F \wedge \eta_F + (-1)^{|F|}e_F \wedge \partial \eta_F = \partial e_F \wedge \eta_F + (-1)^{|F|}e_F \wedge \omega_F,$$
so
$$\omega -\partial((-1)^{|F|}e_F\wedge \eta_F) = \sum_{i=1}^s e_{F\setminus\{a_i\}}  \wedge ((-1)^{|F|+i} \eta_F)+\sum_{S\subseteq X, S\ne F}e_S\wedge \omega_S.$$

Note that $(-1)^{|F|}e_F\wedge \eta_F \in \mathcal C_{i+1}(\Delta;k)$. By repeating this process after finitely many steps, we can find an element $\eta\in \mathcal C_{i+1}(\Delta;k)$ such that $\omega - \partial \eta \in \mathcal C_i(\Delta(G)\setminus X;k)$, and hence $\omega=\partial \eta +\omega'$ for some $\omega'\in C_i(\Delta(G)\setminus X;k)$, as required.
\end{proof}

\begin{lem} \label{A3} Let $G$ be a locally Gorenstein graph in $W_2$. Let $v$ be a vertex and let $H:=G[N(v)]$. Assume that there is an independent set $X$ of $H$ such that for every nonempty independent sets $S$ of $H\setminus X$ we have either $H_S$ has an isolated vertex or $H_S$ is empty. Then $\h_i(\Delta(G);k) = 0$ for all $i < \dim(\Delta(G))$.
\end{lem}
\begin{proof} Let $\Delta := \Delta(G)$ and $d:= \dim(\Delta(G))$. We first claim that 
\begin{equation} \label{EQ-02}
\h_{i-|S|}(\Delta(G_S)\setminus V(H);k) = 0 \ \text{ for any } \emptyset \ne S \in \Delta(H)\setminus X.
\end{equation}

Indeed, observe that $G_S$ is Gorenstein because $G$ is locally Gorenstein. By Lemma $\ref{ML}$ we have $\alpha(G_S) = d - |S|$. If $H_S$ is empty. Since $i-|S| < d - |S| = \dim(\Delta(G_S))$, by Lemma $\ref{Reisner}$ we get $\h_{i-|S|}(\Delta(G_S)\setminus V(H);k) =\h_{i-|S|}(\Delta(G_S);k) =  0$.

If $H_S$ is not empty, then $H_S$ has an isolated vertex, and so $\Delta(G_S)|_{V(H_S)}$ is a cone. By Lemma $\ref{ZV}$ we get $\h_{i-|F|}(\Delta(G_S)\setminus V(H);k) = \h_{i-|F|}(\Delta(G_F\setminus V(H_S));k)=0$, as claimed.

We now prove prove $\h_i(\Delta;k)=0$. It suffices to show that if $\omega\in\mathcal C_i(\Delta;k)$ with $\partial \omega=0$, then $\omega = \partial \zeta$ for some $\zeta\in \mathcal C_{i+1}(\Delta;k)$. 

By Lemma $\ref{A2}$, $\omega = \partial \eta + \omega'$ where $\eta\in C_{i+1}(\Delta;k)$ and
$\omega' \in C_i(\Delta(G)\setminus X;k)$. Since $\omega' \in C_i(\Delta(G)\setminus X;k)$, we can write $\omega'$ as
$$\omega' = \sum_{S \in \Delta(H)\setminus X} e_S\wedge \omega_S$$
where $\omega_S\in \mathcal C_{i-|S|}(\Delta(G_S\setminus V(H));k)$. By using the same argument as in the proof of Lemma $\ref{A2}$ where the equation $(\ref{EQ-01})$ is replaced by the equation $(\ref{EQ-02})$, we have $\eta'\in \mathcal C_{i+1}(\Delta;k)$ such that $$\omega' - \partial \eta' \in \mathcal C_i(\Delta(G)\setminus V(H);k).$$

Note that $\Delta(G)\setminus V(H) = \st_{\Delta}(v)$ and $\h_i(\st_{\Delta}(v);k) = 0$. Since $\partial (\omega' -\partial \eta') = \partial \omega' - \partial ^2 \eta' = 0$, there is $\xi\in \mathcal C_{i+1}(\st_{\Delta}(v);k)\subseteq \mathcal C_{i+1}(\Delta;k)$ such that
$\omega' - \partial \eta' = \partial \xi$, and hence $\omega' = \partial (\eta'+\xi)$.

Together with $\omega = \partial \eta+\omega'$, we obtain $\omega = \partial(\eta+\eta' +\xi)$, as required.
\end{proof}

\begin{rem} If $G[N(v)]$ has an isolated vertex, say $x$, then the set $X =\{x\}$ satisfies all conditions in the lemma $\ref{A3}$.
\end{rem}

A graph is {\it totally disconnected} if there is no path connecting any pair of vertices.

\begin{lem} \label{A4} Let $G$ be a graph with at most $5$ vertices and $\k(\Delta(G)) = 0$. Assume that for any independent set $X$ of $G$ there is a nonempty independent set $S$ of $G\setminus X$ such that neither $G_S$ has an isolated vertex nor $G_S$ is empty. Then $G$ must be either a pentagon with only one chord or a two $3$-cycles with  only one vertex in common.
\end{lem}
\begin{proof} Observe that $G$ has no isolated vertices by the assumption. Let $n := |V(G)|$. We consider four possible cases:

{\it Case 1}: $n = 2$. In this case $G$ is an edge, but then $\k(\Delta(G)) = 1$.

{\it Case 2}: $n =3$. We imply that $G$ is connected since $G$ has no isolated vertices. Thus, $G$ is either a $3$-cycle or a $3$-path, but $\k(\Delta(G)) \ne 0$ in both cases.

{\it Case 3}: $n =4$. If $G$ is disconnected, then $G$ must be two disjoint edges. In this case, $\k(\Delta(G)) = -1$. Hence, we assume that $G$ is connected. Now if $G$ has a $3$-cycle $C_3$, let $x$ be a vertex not lying on $C_3$. Then for every nonempty independent set $S$ of $G\setminus x$, we see that $G_S$ is either just $x$ or empty. Thus, $G$ has no any $3$-cycles, and thus either $G$ has a $4$-cycle or $G$ is a tree. If $G$ has a $4$-cycle, then $G_S$ is just one vertex or empty for all nonempty independent set $S$ of $G$. If $G$ is a tree, then let $X$ be the set of leaves of $G$. Observe that $|X| \geqslant 2$, so that $G_S$ is either empty or totally disconnected for any non-empty independent set $S$ of $G\setminus X$.

{\it Case 4}: $n = 5$. If $G$ is disconnected, then $G$ has just two components of size $2$ and $3$. Therefore, $G$ is either an edge and a triangle or and edge and a $3$-path. In both cases, we can see that $\k(\Delta(G))\ne 0$. Hence, we assume that $G$ is connected. Let $V(G) =\{a,b,c,d,e\}$.

If $G$ has a $5$-cycle $C_5$, so that $\alpha(G) \leqslant 2$. Note that $\k(\Delta(C_5)) = -1$, so $G$ must have at least one chord. If $G$ has only one chord, then $\k(G)=0$. If $G$ has at least two chords, then $f_1(\Delta(G)) \leqslant 3$, and then $\k(\Delta(G)) = -1+5-f_1(\Delta(G)) > 0$. It follows that $G$ have only one chord.

If $G$ has no any $5$-cycles, but it has a $3$-cycle $C_3$. Assume that $E(C_3) = \{ab,bc,ca\}$. Since $G$ is connected, we may assume that $d$ is adjacent to $a$. If $e$ is not adjacent to $d$, then $\{e,d\}$ is an independent set of $G$. In this case, $G_S$ is totally disconnected for any nonempty independent set $S$ of $G\setminus\{d,e\}$. Thus, $e$ is must be adjacent to $d$. Since $G$ has no $5$-cycles, $e$ is adjacent to neither $b$ nor $c$. Now if $e$ is not adjacent to $a$, then we have
$\k(\Delta(G)) = 1$. Hence, $e$ is adjacent to $a$, and hence $G$ is two $3$-cycles with only one vertex in common. Note that $\k(\Delta(G))=0$ for that graph $G$.

If $G$ has neither $3$-cycles nor $5$-cycles, then either $G$ has only $4$-cycles or $G$ is a tree. If $G$ has a $4$-cycle, say $C_4$. We may assume that $E(C_4) = \{ab,bc,cd,da\}$. Since $G$ is connected, we may assume that $e$ is adjacent to $a$. Since $G$ has neither $3$-cycles nor $5$-cycles, the edge set of $G$ is just $\{ab,bc,cd,da,ae\}$. Then, $\{c,e\}$ is an independent subset of $G$ and $G_S$ are totally disconnected for any nonempty independent set $S$ of $G\setminus\{d,e\}$.

Finally, $G$ is a tree. If $G$ has exactly two leaves, then $G$ must be the $5$-path. But then $\k(\Delta(G)) = -1$, hence $G$ has at least three leaves. Let $X$ be the set of leaves of $G$, in particular, $X$ is  an independent set of $G$. Since $G\setminus X$ is connected and has at most two vertices, therefore $G_S$ is totally-disconnected for all nonempty independent sets $S$ of $G\setminus X$.

In summary, $G$ is either a $5$-cycle with only one chord or a two $3$-cycles with only one vertex in common.
\end{proof}

A {\it pure} simplicial complex $\Delta$ of dimension $d$ is called {\it connected in codimension one} if for any two facets $F$ and $H$ of $\Delta$, there exists a sequence of facets $F = F_0, F_1, \cdots, F_{p-1}, F_p = H$ such that $|Fi \cap F_{i+1}| = d$ for each $i$. By using the same argument as in the proof of \cite[Theorem $9.1.12$]{HH2} we obtain.

\begin{lem} \label{C1} Let $G$ be a well-covered graph with $\alpha(G)\geqslant 2$. If $G$ is locally Cohen-Macaulay and $\Delta(G)$ is connected, then $\Delta(G)$ is connected in codimension one.
\end{lem}

We now in a position to prove the main result of this paper.

\begin{thm}\label{GlobalVersion} Let $G$ be a pseudo-planar graph without isolated vertices. Then $G$ is Gorenstein if and only if $\Delta(G)$ is Eulerian.
\end{thm}
\begin{proof} If $G$ is Gorenstein. As $G$  has no isolated vertices, we have $\Delta(G)=\core(\Delta(G))$, so $\Delta(G)$ is Eulerian by Lemma $\ref{Stanley}$.

Conversely, Assume $\Delta(G)$ is Eulerian. By Lemma $\ref{E1}$ we have $G$ is in $W_2$.  We now prove by induction on $\alpha(G)$ that $G$ is Gorenstein. If $\alpha(G) \leqslant 2$, the assertion follows from Lemma $\ref{A1}$. 

Assume that $\alpha(G)\geqslant 3$. Let $\Delta := \Delta(G)$ and $d := \dim(\Delta) = \alpha(G)-1$. Since $\Delta$ is Eulerian, by Lemma $\ref{Stanley}$ we have $\Delta$ is Gorenstein if $\Delta$ is Cohen-Macaulay. Now for every vertex $x$ of $G$ we have $G_x$ is well-covered with $\alpha(G_x) =\alpha(G)-1$ by Lemma $\ref{LC}$. As $\Delta(G_x)$ is Eulerian, by the induction hypothesis we have $G_x$ is Gorenstein. This yields $G$ is locally Gorenstein. Thus,  By Lemma $\ref{Reisner}$ we have $\Delta$ is Cohen-Macaulay if
\begin{equation} \label {ME} \h_i(\Delta;k) = 0 \ \text{ for all } i < d. \end{equation}

Let $v$ be a vertex of $G$ such that $\delta^*(v)$ is minimal and let $H := G[N(v)]$. If $H$ has an isolated vertex, then the equation $(\ref{ME})$ holds by Lemma $\ref{A3}$. Therefore, we assume that $H$ has no isolated vertices, and therefore $H$ has at most $5$ vertices by the pseudo-planar property of $G$.

Since $\Delta$ is Eulerian, by Lemma $\ref{E2}$ we have $\k(\Delta(H)) = 0$. By Lemmas $\ref{A3}$ and $\ref{A4}$, the equation $(\ref{ME})$ holds except for $H$ is either a pentagon with only one chord or a two $3$-cycles with only one vertex in common. We next prove the equation $(\ref{ME})$ holds for such two exceptional graphs $H$.

If $H$ is a pentagon with one chord, assume that $V(H) =\{x,y,z,s,t\}$ and $E(H) = \{xy,yz,zs,st,tx, ys\}$ (see Figure $1$). Let $\omega \in \mathcal C_i(\Delta(G);k)$ such that $\partial \omega = 0$. We will prove that $\omega$ is a boundary. Let $X := V(H)$. First we claim that
\begin{equation}\label{G1}
\omega = \partial \zeta + e_s \wedge \omega_s + \omega',
\end{equation}
where $\zeta\in \mathcal{C}_{i+1}(\Delta;k)$,  $\omega_s \in \mathcal{C}_{i-1}(\Delta(G_s)\setminus X)$ and $\omega' \in \mathcal{C}_{i}(\Delta\setminus X)$.

\medskip

\begin{center}
\includegraphics[scale=0.7]{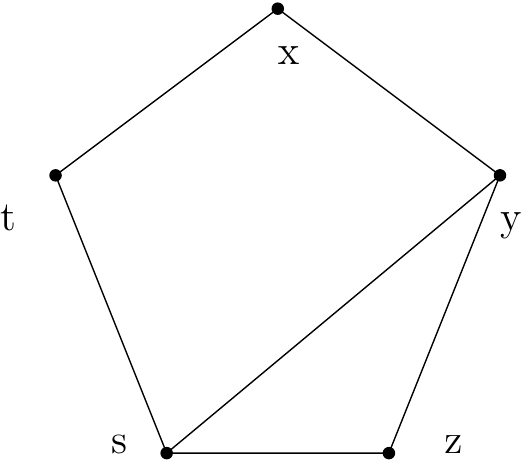}\\
\medskip
{\it Figure 1. A pentagon with one chord.}
\end{center}

\medskip

Indeed, since $\{z,t\}$ is an independent set of $G$, by Lemma $\ref{A2}$ we can write $\omega$ as $\omega = \partial \eta + \psi$ where $\eta\in\mathcal C_{i+1}(\Delta;k)$ and $\psi \in\mathcal C_i(\Delta\setminus\{z,t\};k)$ with $\partial \psi = 0$. We next write $\psi$ as
\begin{equation}\label{G2} 
\psi = e_x \wedge e_s \wedge \psi_{xs} + e_x \wedge \psi_x+ e_y \wedge \psi_y + e_s \wedge \psi_s+ \psi'
\end{equation}
where 
$$\psi_{xs}\in \mathcal C_{i-2}(\Delta(G_{\{x,s\}});k); \psi_x, \psi_y, \psi_s \in \mathcal C_{i-1}(\Delta\setminus X;k); \text{ and } \psi' \in \mathcal C_i(\Delta\setminus X;k).$$ 

As $\partial \psi = 0$, from $(\ref{G2})$ we imply that $\partial \psi_{xs} = 0$. Since $G$ is well-covered and locally Gorenstein, $G_{\{x,s\}}$ is Gorenstein with $$\dim(\Delta(G_{\{x,s\}})) = \dim(\Delta(G))-2=d-2.$$

Since $i-2 < \dim(\Delta(G_{\{x,s\}}))$, $\h_{i-2}(\Delta(G_{\{x,s\}});k) = 0$ by  Lemma $\ref{Reisner}$. This gives $\psi_{xs} =\partial \eta_{xs}$ for some $\eta_{xs}\in \mathcal C_{i-1}(\Delta(G_{\{x,s\}});k)$. Together with $(\ref{G2})$ we have
$$ \psi - \partial (e_x\wedge e_s \wedge \eta_{xs}) = e_x \wedge (\psi_x + \eta_{xs}) +e_y\wedge \psi_y+ e_s \wedge (\psi_s-  \eta_{xs}) +  \psi'.$$

Let $\phi_x:= \psi_x+ \eta_{xs}$ and $\phi_s := \psi_s -  \eta_{xs}$. Then $\phi_x, \phi_s\in \mathcal C_{i-1}(\Delta\setminus X;k)$ and
$$\psi - \partial (e_x\wedge e_s \wedge \eta_{xs})= e_x \wedge \phi_x +e_y\wedge \psi_y+ e_s \wedge \phi_s +  \psi'.$$
Combining with $\partial (\psi - \partial (e_x\wedge e_s \wedge \psi_{xs})) = 0$, we get $\partial \psi_y = 0$. Note that $t$ is an vertex of $G_y$ and $\Delta(G_y)$ is Gorenstein with $\core(\Delta(G_y)) =\Delta(G_y)$. By Lemma $\ref{ZV}$ we have $\h_{i-1}(\Delta(G_y)\setminus t;k)=0$. Therefore,  $\psi_y = \partial \omega_y$ for some $\omega_y\in \mathcal C_i(\Delta(G_y)\setminus t;k)$. This yields
\begin{equation}\label{G3}
\psi +\partial (-e_x\wedge e_s \wedge \eta_{xs}+ e_y\wedge\omega_y) = e_x \wedge \phi_x + e_s \wedge \phi_s + (\omega_y +\psi').
\end{equation}

Since $\omega_y +  \psi' \in \mathcal{C}_{i}(\Delta\setminus X;k)$ and $\partial(\psi + \partial (-e_x\wedge e_s \wedge \psi_{xs}+e_y\wedge\omega_y)) = 0$, from $(\ref{G3})$ we get $\partial \phi_x = 0$. Because $z$ is a vertex of $G_x$, $\h_{i-1}(\Delta(G_x)\setminus \{z\};k)=0$ by Lemma $\ref{ZV}$. Note that $\phi_x \in \mathcal C_{i-1}(\Delta(G_x)\setminus \{z\};k)$, so $\phi_x = \partial \omega_x$ for some $\omega_x \in \mathcal C_{i}(\Delta(G_x)\setminus \{z\};k)$. Together with $(\ref{G3})$ we have
\begin{equation}\label{G4}
\psi +\partial (-e_x\wedge e_s\wedge \eta_{xs}+ e_x \wedge \omega_x+ e_y\wedge\omega_y) = \omega_x + e_s \wedge \phi_s + (\omega_y +\psi').
\end{equation}

Observe that $\omega_x \in \mathcal C_{i}(\Delta(G_x)\setminus \{z\};k) \subseteq \mathcal C_{i}(\Delta\setminus (X \setminus \{s\});k)$, thus we can write
$$\omega_x = e_s \wedge \eta_s + \eta'$$ 
where $\eta_s \in C_{i-1}(\Delta(G_s)\setminus X;k)$ and $\eta' \in C_{i}(\Delta\setminus X;k)$. By substituting it into the right hand sinde of $(\ref{G4})$ we get
\begin{equation}\label{G5}
\psi +\partial (-e_x\wedge e_s\wedge \eta_{xs}+ e_x \wedge \omega_x+ e_y\wedge\omega_y) = e_s \wedge (\phi_s+\eta_s) + (\omega_y +\psi'+\eta').
\end{equation}

Let $\zeta := \eta+ e_x\wedge e_s \wedge \eta_{xs} - e_x\wedge \omega_x - e_y \wedge \omega_y$, $\omega_s := \phi_s+\eta_s$ and $\omega' := \omega_y +\psi'+\eta'$. Together with $(\ref{G2})$ we obtain $\omega = \partial \eta + e_s \wedge \omega_s + \omega'$, and the claim $(\ref{G1})$ follows.

We now go back to prove that $\omega$ is a boundary. Together $\partial \omega = 0$ with $(\ref{G1})$ we get $\partial \omega_s = 0$. Since $G_s$ is Gorenstein without isolated vertices and that $x$ is a vertex of $G_s$, by Lemma $\ref{ZV}$ we get $H_{i-1}(\Delta(G_s)\setminus\{x\}; k) = 0$. As $\omega_s  \in C_{i-1}(G_s\setminus \{x\})$, $\omega_s = \partial \gamma_s$ for some $\gamma_s \in \mathcal C_{i}(G_s\setminus \{x\};k)$. Together with $(\ref{G1})$, this yields
\begin{equation}\label{G6}
\omega = \partial(\eta - e_s \wedge \gamma_s) + (\gamma_s + \omega').
\end{equation}

Since $\partial \omega = 0$, from $(\ref{G6})$ we have $\partial(\gamma_s + \omega') =0$. Note that $\Delta\setminus X = \st_{\Delta}(v)$, so $\h_i(\Delta\setminus X;k)=0$. Since $\gamma_s + \omega'\in \mathcal C_i(\Delta\setminus X;k)$, $\gamma_s+\omega' = \partial \gamma$ for some $\gamma \in \mathcal C_{i+1}(\Delta\setminus X;k) \subseteq \mathcal C_{i+1}(\Delta;k)$. This gives $\omega = \partial(\eta - e_s \wedge \gamma_s+\gamma)$, so $\omega$ is a boundary. Thus, $\h_i(\Delta;k)=0$.

For the remain case, $H$ is a two $3$-cycles with only one vertex in common, assume that $V(H) =\{x,y,z,s,t\}$ and $E(H) = \{xy,yz,zx, ys, st, ty\}$ (see Figure $2$). We first prove that $\h_i(\Delta;k)=0$ for all $i < d-1$. Let $\omega \in \mathcal C_i(\Delta(G);k)$ such that $\partial \omega = 0$. By the same argument as in the claim $(\ref{G1})$ we can write $\omega$ as
\begin{equation}\label{G7}
\omega = \partial \zeta + e_s \wedge \omega_s + \omega',
\end{equation}
where $\zeta\in \mathcal{C}_{i+1}(\Delta;k)$,  $\omega_s \in \mathcal{C}_{i-1}(\Delta(G_s)\setminus X)$ and $\omega' \in \mathcal{C}_{i}(\Delta\setminus X)$.

\medskip

\begin{center}
\includegraphics[scale=0.7]{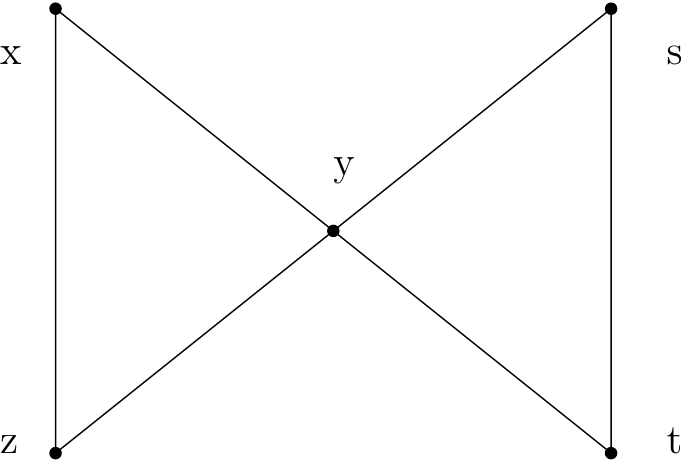}\\
\medskip
{\it Figure 2. Two triangle with one vertex in common.}
\end{center}

\medskip

Since $\partial \omega = 0$, from $(\ref{G7})$ we deduce that $\partial \omega_s = 0$. Since $z$ is a vertex of $G_s$,  $\h_{i-1}(\Delta(G_s)\setminus z;k) = 0$ by Lemma $\ref{ZV}$. Note that $\omega_s \in \mathcal C_{i-1}(\Delta(G_y)\setminus z;k)$, therefore $\omega_s= \partial \gamma_s$ for some $\gamma_s\in C_i(\Delta(G_s)\setminus z;k)$. Thus,
$$\omega = \partial (\zeta-e_s\wedge \gamma_s) +\gamma_s +\omega'.$$

Note that $V(G_s \setminus z) \cap X =\{x\}$, so we can write $\gamma_s = e_x \wedge \psi_x + \phi$ where $\psi_x \in \mathcal C_{i-1}(\Delta(G_{\{x,s\}});k)$ and $\phi\in\mathcal C_i(\Delta(G)\setminus X;k)$. Then,
\begin{equation}\label{G8}
\omega = \partial (\zeta-e_s\wedge \gamma_s) +e_x\wedge \psi_x+  \phi +\omega'.
\end{equation}
Together with $\partial \omega = 0$, this expression implies $\partial \psi_x = 0$.

On the other hand, $\Delta(G_{\{x,s\}})$ is Gorenstein of dimension $d-2$ and $i - 1  < d -2$, thus $\h_{i-1}(\Delta(G_{\{x,s\}});k)=0$ by Lemma $\ref{Reisner}$. Hence,  $\psi_x =\partial \phi_x$ for some $\phi_x\in \mathcal C_i(\Delta(G_{\{x,s\}});k)$. Together with $(\ref{G8})$ we obtain
$$\omega = \partial (\zeta-e_s\wedge \gamma_s-e_x \wedge \phi_x)  +(\phi-\phi_x+ \omega').$$

Combining with $\partial \omega = 0$ we have $\partial(\phi-\phi_x+ \omega') = 0$. On the other hand, as $\phi-\phi_x+ \omega' \in\mathcal C_i(\Delta\setminus X;k)$ and $\h_i(\Delta\setminus X;k) = 0$, we have $\phi-\phi_x+ \omega' = \partial \gamma$ for some $\gamma \in C_{i+1}(\Delta\setminus X;k)$. This gives $\omega = \partial (\zeta-e_s\wedge \gamma_s-e_x \wedge \phi_x+\gamma)$, so $\omega$ is a boundary. It follows that $\h_i(\Delta;k)=0$ for all $i < d-1$.

Finally we prove $\h_{d-1}(\Delta;k)=0$. Since $\h_i(\Delta;k)=0$ for all $i < d-1$, the Euler characteristic of $\Delta$ now becomes
$$\k(\Delta) = \sum_{i=-1}^d (-1)^i \dim_k \h_i(\Delta;k) = (-1)^{d-1} \dim_k \h_{d-1} (\Delta;k) + (-1)^d \dim_k \h_d (\Delta;k).$$
Note that $\k(\Delta) = (-1)^d$, so $\dim_k \h_{d-1}(\Delta;k) = \dim_k \h_d(\Delta;k)-1$. On the other hand, $\h_0(\Delta;k) =0$ because $0 < d-1$, therefore $\Delta$ is connected. As $G$ is  locally Gorenstein, by Lemma $\ref{C1}$ we have $\Delta$ is connected in codimension one. Now let $F$ be any face of $\Delta$ with $\dim F = d-1$ so that $|F| = \alpha(G)-1$.  Since $\Delta(G_F)$ is Eulerian with $\alpha(G_F) = 1$, $G_F$ is just one edge of $G$, say $uv$. This means that $F$ is the common face of exactly two facets $F\cup \{u\}$ and $F\cup \{v\}$ in $\Delta$. Hence, $\Delta$  is a $d$-dimensional {\it pseudomanifold} without boundary (see \cite[Definition $3.15$]{S}), and hence $\dim \h_d(\Delta;k) \leqslant 1$ by \cite[Proposition $3.16$]{S}.
Consequently, $\dim_k \h_{d-1}(\Delta;k) = \dim_k \h_d(\Delta;k)-1 \leqslant 0$. It follows that $\dim_k \h_{d-1}(\Delta;k) = 0$, so $\h_{d-1}(\Delta;k)=0$. The proof is complete.
\end{proof}

Since a planar graph is a pseudo-planar graph, by using Theorem $\ref{GlobalVersion}$ we obtain a graph-theoretical characterization of Gorenstein planar graphs. Note that a characterization of Gorenstein triangle-free planar graph is obtained in \cite{HMT}.

\begin{cor} A planar graph is Gorenstein if and only if its independence complex is Eulerian.
\end{cor}

We conclude this paper with a simple construction of Gorenstein graphs and Gorenstein pseudo-planar graphs from the given ones (it is similar to Pinter \cite{Pi}). 

\begin{prop} \label{CT2} Let $H$ be a Gorenstein graph and let $x$ be a non isolated vertex of $H$. Let the three new points be $a,b$ and $c$. Join $a$ to $b$ and every neighbor of $x$; join $b$ to $c$; and  join $c$ to $x$. Let $G$ be the graph obtained from this construction. Then,
\begin{enumerate}
\item $G$ is a Gorenstein graph with $\alpha(G) = \alpha(H)+1$.
\item $G$ is pseudo-planar if $H$ is pseudo-planar.
\end{enumerate}
\end{prop} 
\begin{proof} We may assume that $H$ has no isolated vertices so that $H$ is in $W_2$. We prove the proposition by induction on $\alpha(H)$. If $\alpha(H) = 1$, then $H = K_2$. In this case, $G$ is just a pentagon, so the proposition holds.

Assume that $\alpha(H) \geqslant 2$. We first prove $G$ is Gorenstein. Let $v$ be an arbitrary vertex of $G$. We now claim that $G_v$ is Gorenstein with $\alpha(G_v)=\alpha(H)$ and without isolated vertices. Indeed, observe that if $v\notin \{a,b,c,x\}\cup N_H(x)$ then $v\in V(H_x)$. We consider the four possible cases:

{\it Case 1:} $v\in\{b,c\}$. $G_v \cong H$, and the claim holds.

{\it Case 2:} $v \in\{a,x\}$. $G_v$ is isomorphic to the disjoint union of $H_x$ and $K_2$, and the claim holds

{\it Case 3:} $v \in N_H(x)$. $G_v$ is the disjoint union of $H_v$ and the edge $bc$, and the claim holds

{\it Case 4:} $v\in V(H_x)$. In this case, $x$ is a vertex of $H_v$. Since $H$ is in $W_2$, $x$ is not an isolated vertex in $H_v$. In this case $G_v$ is exactly the graph obtained by the construction in the lemma that starts with $H_v$. Since $\alpha(H_v)=\alpha(H)-1$, by the induction hypothesis we have $G_v$ is Gorenstein with $\alpha(G_v) =\alpha(H_v)+1 = \alpha(H)$. Notice that $G_v$ has no isolated vertices. This proves the claim.

We now turn to prove that $G$ is Gorenstein. By the claim, $G$ is semi-Euler and locally Gorenstein.  Since $G[N_G(b)]$ is just two isolated vertices $a$ and $c$, we have $\k(\Delta(G[N_G(b)]) = 0$. By Corollary $\ref{E2}$ we have $\Delta(G)$ is Eulerian. By Lemmas $\ref{Reisner}$ and $\ref{Stanley}$, $G$ is Gorenstein if $\h_i(\Delta(G);k) = 0$ for $i <\dim(\Delta(G))$. This follows from Lemma $\ref{A3}$ since $G[N_G(b)]$ is totally-disconnected. Thus, $G$ is Gorenstein. Clearly, $\alpha(G) = \alpha(G_b)+1 = \alpha(H)+1$.

Assume that $H$ is pseudo-planar. We next prove that $G$ is pseudo-planar as well. Let $S$ be arbitrary independent set of $G$. We need to prove that:
\begin{enumerate}
\item $\delta^*(G_S) \leqslant 5$; and
\item $G_S \not \cong 2K_2 * 2 K_2$.
\end{enumerate}

Observe that each vertex of $2K_2 * 2 K_2$ has degree $5$. Since $\deg_G(b) = \deg_G(c)=2$, two conditions above satisfy if  either $b$ or $c$ is a vertex of $G_S$.

Hence, we may assume that both $b$ and $c$ are not vertices of $G_S$. It follows that either $a, x\in S$ or $S \cap \{b,c\} \ne \emptyset$.

If $a,x\in S$, then let $S' := S\setminus \{a\}$. Observe that $S'\in\Delta(H)$ and $G_S = H_{S'}$, hence two conditions satisfy in this case.

If either $b$ or $c$ in $S$. By symmetry in $G$, we may assume that $b\in S$. Let $S' := S\setminus\{b\}$. Then, $S'\in\Delta(H)$ and $G_S = H_{S'}$, hence two conditions also satisfy in this case. The proof of the proposition is complete.
\end{proof}

\begin{exm} Start with $R_3$ (the name is due to \cite{PL}), we have the following Gorenstein pseudo-planar graph (see Figure $3$). Note that the obtained graph  is not planar because it has a minor $K_5$.

\medskip

\begin{center}

\includegraphics[scale=0.7]{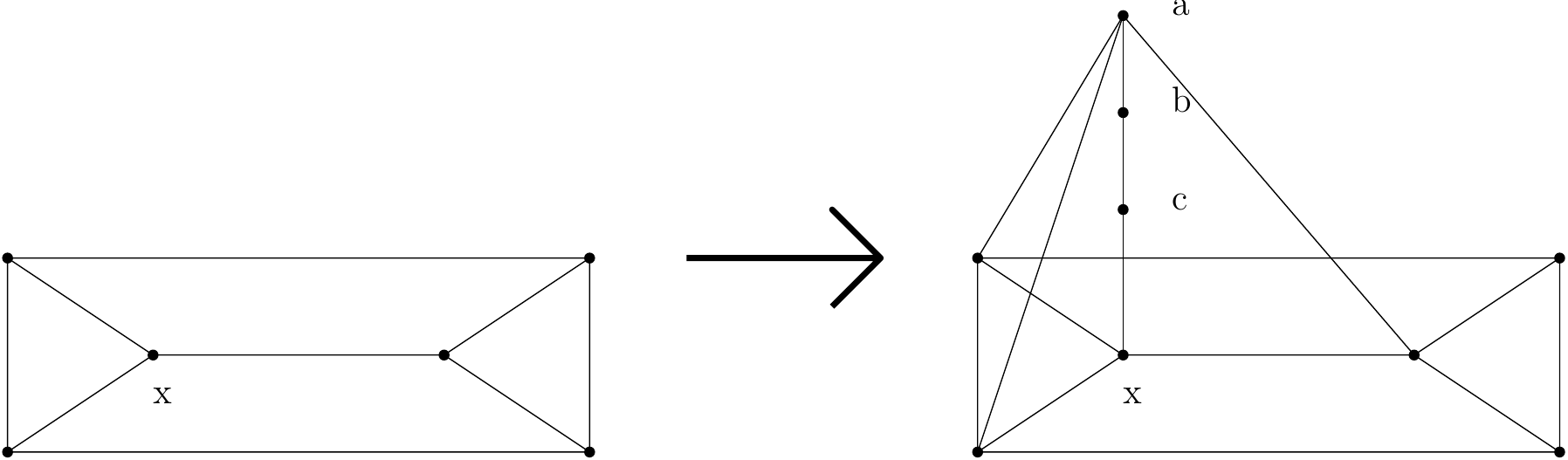}\\

\medskip

{\it Figure 3. The Gorenstein pseudo-planar graph obtained from $R_3$.}
\end{center}
\end{exm}

\subsection*{Acknowledgment} This work is partially supported by NAFOSTED (Vietnam) under the grant number 101.04 - 2015.02.


\end{document}